\documentclass[pdflatex,sn-mathphys-num]{sn-jnl}

\usepackage{graphicx}%
\usepackage{multirow}%
\usepackage{amsmath,amssymb,amsfonts}%
\usepackage{amsthm}%
\usepackage{mathrsfs}%
\usepackage[title]{appendix}%
\usepackage{xcolor}%
\usepackage{textcomp}%
\usepackage{manyfoot}%
\usepackage{booktabs}%
\usepackage{algorithm}%
\usepackage{algorithmicx}%
\usepackage{algpseudocode}%
\usepackage{listings}%

\theoremstyle{thmstyleone}%
\newtheorem{theorem}{Theorem}
\newtheorem{proposition}[theorem]{Proposition}%

\theoremstyle{thmstyletwo}%

\theoremstyle{thmstylethree}%
\newtheorem{definition}{Definition}%

\begin{document}

\title[five-dimensional naturally graded associative algebras]{The relationship between local derivations and local automorphisms of some associative algebras}


\author*[1,2]{\fnm{Farkhodzhon} \sur{Arzikulov}}\email{arzikulovfn@gmail.com}

\author[2]{\fnm{Utkir} \sur{Khakimov}}\email{khakimov$_{-}$u@inbox.ru}
\equalcont{These authors contributed equally to this work.}

\author[3]{\fnm{Abduqaxxor} \sur{Qurbonov}}\email{qabduqaxxor@mail.ru}
\equalcont{These authors contributed equally to this work.}

\affil*[1]{\orgdiv{V.I. Romanovskiy Institute of Mathematics}, \orgname{Uzbekistan Academy of Sciences}, \orgaddress{\street{Olmazor district, University street 9}, \city{Tashkent}, \postcode{100174}, \country{Uzbekistan}}}

\affil[2]{\orgdiv{Mathematics}, \orgname{Andijan State University}, \orgaddress{\street{University 129}, \city{Andijan}, \postcode{170100}, \country{Uzbekistan}}}

\affil[3]{\orgdiv{Mathematics}, \orgname{Namangan State University}, \orgaddress{\street{Boburshoh street 161}, \city{Namangan}, \postcode{160107}, \country{Uzbekistan}}}


\abstract{In the present paper local derivations and local automorphisms of five-dimensional naturally graded nilpotent associative algebras are studied. Namely, a general form of the matrices of local derivations and local automorphisms of algebras $\pi_2$ and $\pi_3$ is clarified. It turns out that the general form of the matrix of an automorphism (derivation) on these algebras does not coincide with the local automorphism's (resp. local derivation's) matrix's general form on these algebras. Therefore, these associative algebras have local automorphisms (resp. local derivations) that are not automorphisms (resp. derivations). We also
establish a relationship between local automorphisms and local derivations via an exponential expression.
We prove that the sets of local derivations of algebras $\pi_2$ and $\pi_3$ form Lie algebras with respect to the Lie brackets. Thus,
we show that the Lie algebra problem from the Ayupov-Eldique-Kudaybergenov problems for local derivations of the algebras under consideration has a positive solution. The remaining problems from the Ayupov-Eldique-Kudaybergenov problems also have a positive solution for algebras $\pi_2$ and $\pi_3$.

}

\keywords{associative algebra, automorphism, local automorphism, nilpotent algebra, naturally graded algebra}


\pacs[MSC Classification]{16W25, 16W20}

\maketitle

\section{Introduction}

	The concept of local derivations was introduced by R. Kadison in his paper \cite{8}. In this paper R. Kadison proved that each continuous local derivation from a von Neumann algebra into its dual Banach bemodule is a derivation. In \cite{7} B.Jonson  extended the above result by proving that every local derivation from a $C^*$-algebra into its Banach bimodule is a derivation. Based on these results, many authors have studied local derivations on operator algebras.
	Present paper is mainly devoted to the study of local automorphisms on associative algebras. The following results were obtained by the authors on local automorphisms: in \cite{3} Sh.A.Ayupov, K.K.Kudaybergenov, B.A.Omirov considered automorphisms of simple Leibniz algebras. They established that every 2-local automorphism on a complex finite-dimensional simple Leibniz algebra is an automorphism and that nilpotent Leibniz algebras admit 2-local automorphisms which are not automorphisms. A similar problem concerning local automorphism on simple Leibniz algebras is reduced to the case of simple Lie algebras.

Also, local automorphisms of certain finite-dimensional simple Lie and Leibniz algebras are investigated in \cite{2}. T.Becker, J.Escobar, C.Salas, and R.Turdibaev in \cite{5} proved that the set of local automorphisms $LAut(sl_2)$ coincides with the group $Aut^{\mp}sl_2$ of all automorphisms and anti-automorphisms. Later in \cite{6} M.Costantini proved that a linear map on a simple Lie algebra is a local automorphism if and only if it is either an automorphism or an anti-automorphism. Similar results concerning 2-local automorphisms on Lie superalgebras were obtained in \cite{9}. F.N.Arzikulov, I.A.Karimjanov, and S.M.Umrzaqov established that every local and 2-local automorphisms on the solvable Leibniz algebras with null-filiform and naturally graded non-Lie filiform nilradicals, whose dimension of complementary space is maximal, is an automorphism \cite{1}. Recently, local automorphisms of Cayley algebras were considered in \cite{AEK}.

In the present paper automorphisms, derivations, local automorphisms and local derivations of five-dimensional naturally graded associative algebras $\pi_2$ and $\pi_3$ with characteristic sequence $C(A)=(3,2)$, defined in section \ref{sec1}, are studied. Namely, general forms of the matrices of local derivations and local automorphisms of these algebras are clarified. It turns out that the general form of the matrix of a local automorphism (local derivation) on these algebras does not coincide with the automorphism's (resp. derivation's) matrix's general form. Therefore, these associative algebras have local automorphisms (local derivations) that are not automorphisms (derivations).

Other problem under consideration in the present paper is conjectured in \cite{AEK} by Ayupov, Eldique and Kudaybergenov.
They have proved that the group $LocAut(\mathcal{C})$ of all local automorphisms of a Cayley algebra $\mathcal{C}$ is a Lie group and
its Lie algebra $Lie(LocAut(\mathcal{C}))$ is isomorphic to the Lie algebra $LocDer(\mathcal{C})$ of all local derivations of
$\mathcal{C}$, i.e., $Lie(LocAut(\mathcal{C}))\cong LocDer(\mathcal{C})$.
These results leaded them formulate the following problems:

{\bf Problems 1.}
Let $\mathcal{A}$ be a finite dimensional (not necessary associative) algebra over a field ${\Bbb F}$.

(1) Is the group $LocAut(\mathcal{A})$ of all local automorphisms of $\mathcal{A}$ a Lie group?

(2) Is the vector space $LocDer(\mathcal{A})$ of all local derivations of
$\mathcal{A}$ a Lie algebra with respect to the brackets $[\,\, , \,\, ]$?

(3) If the above two assertions are true,
are the Lie algebras $Lie(LocAut(\mathcal{A}))$ and $LocDer(\mathcal{A})$ isomorphic?

It is well known that the space of all derivations $Der(\mathcal{A})$ is a Lie algebra with respect to the Lie bracket. At the same time, it is not clear whether the space of all local derivations $LocDer(\mathcal{A})$ forms a Lie algebra. As in \cite{AEK} the authors have noted for Cayley algebras,
Problems 1
have a positive solution.

In the present paper we show that in the cases of the algebras under consideration
Problems 1 also have a positive solution.

As is known every derivation $D$ of a finite-dimensional algebra $\mathcal{A}$ defines an automorphism
$\phi$ via an exponential expression, i.e., $\phi=I+\sum_{n=1}^\infty \frac{D^n}{n!}$.
We establish a similar relationship between local automorphisms and local derivations of algebras $\pi_2$ and $\pi_3$
via an exponential expression.

\section{Preliminaries}  \label{sec1}

For an algebra $\mathcal{A}$ of an arbitrary variety, we consider the series
\[
\mathcal{A}^1,~~~\mathcal{A}^{i+1}=\sum^{i}_{k=1}\mathcal{A}^k\mathcal{A}^{i-k+1},~~~~i\geq 1.
\]
We say that an algebra $\mathcal{A}$ is \emph{nilpotent} if $\mathcal{A}^i=0$ for some $i\in\mathbb{N}.$ The smallest integer satisfying $\mathcal{A}^i=0$ is called the index of nilpotency or nilindex of $\mathcal{A}.$

An algebra $\mathcal{A}$ is called $N$-graded if there is a decomposition $\mathcal{A}=\bigoplus_{i\in \mathbb{N}}\mathcal{A}_i,\mathcal{A}_i\mathcal{A}_j\subseteq \mathcal{A}_{i+j}$, for all $i,j\in \mathbb{N}$. Given an associative algebra $\mathcal{A}$, one can consider its associated graded associative algebra $gr\mathcal{A}=\bigoplus_{i\in \mathbb{N}} (\mathcal{A}_i/\mathcal{A}_{i+1})$ with the product
$(x+\mathcal{A}_{i+1})(y+\mathcal{A}_{j+1})=xy+\mathcal{A}_{i+j+1},x\in \mathcal{A}_i,y\in \mathcal{A}_j.$

\begin{definition}
Given a nilpotent associative algebra $\mathcal{A}$, we have a natural grading on $\mathcal{A}$ induced by the series $\mathcal{A}_i, \mathcal{A}_i=\mathcal{A}_i/\mathcal{A}_{i+1}, 1\leq i\leq k-1,$ and $gr\mathcal{A}=\mathcal{A}_1\bigoplus \mathcal{A}_2\bigoplus...\bigoplus \mathcal{A}_k.$ If $gr\mathcal{A}$ and $\mathcal{A}$ are isomorphic, denoted by $gr\mathcal{A}\cong \mathcal{A},$ we say that the algebra $\mathcal{A}$ is naturally graded. For any element $x$ in $\mathcal{A}$, we define the left multiplication operator as: $L_x:\mathcal{A}\rightarrow \mathcal{A},z\rightarrow {xz},z\in \mathcal{A}.$ For an element $x\in \mathcal{A}\setminus \mathcal{A}_2,$ define the decreasing sequence $C(x)={(n_1,n_2,...,n_k)}$ which consists of the dimensions of the Jordan blocks of the left multiplication operator $L_x.$ Endow the set of these sequences with the lexicographic order, ${i.e.,} {C(x)}={(n_1,n_2,...,n_k)}\leq C(y)={(m_1,m_2,...,m_s)},$ which means that there is an index $i\in \mathbb{N},$ such that $n_j=m_j$ for all $j< m_i.$
\end{definition}

\begin{definition}
The sequence $C(\mathcal{A})={max}_{x\in \mathcal{A}\setminus {\mathcal{A}_2}}C(x)$ is defined to be the characteristic sequence of the algebra $\mathcal{A}$.
\end{definition}

We assume throughout this section that ${\Bbb F}$ is the field of complex numbers. Let $\mathcal{A}$ be a naturally graded $n$-dimensional quasi-filiform associative algebra. Then, there are two possibilities for the characteristic sequence, either $C(\mathcal{A})={(n - 2,1,1)}$ or  ${C(\mathcal{A})=(n-2,2)}.$

The associative algebras
\begin{equation}\label{eq}\nonumber
\pi_2:\left\{%
\begin{array}{llll}
    {e_1e_1=e_2}, \\
    {e_1e_2=e_2e_1=e_3}, \\
    {e_1e_4=e_4e_1=e_5}, \\
    {e_4e_4=e_5}
\end{array}%
\right., \,\,\,
\pi_3:\left\{%
\begin{array}{llll}
    {e_1e_1=e_2}, \\
    {e_1e_2=e_2e_1=e_3}, \\
    {e_1e_4=e_5}, \\
    {e_4e_4=e_5}
\end{array}%
\right.
\end{equation}
are five-dimensional naturally graded associative algebras with characteristic sequence $C(\mathcal{A})=(3,2)$ {(see Theorem 4.5 in \cite{KL})}.

Now we give a description of automorphisms of five-dimensional  naturally graded associative algebras $\pi_2$ and $\pi_3$ with characteristic sequence $C(\mathcal{A})=(3,2)$.

\begin{definition}
Let $\mathcal{A}$ be an algebra. A bijective linear map $\Phi:A\rightarrow A$ is called an automorphism, if, for any elements $x,y\in \mathcal{A},~\Phi(xy)=\Phi(x)\Phi(y).$
\end{definition}

In what follows we will need the following theorems.

\begin{theorem} \label{2.4} \cite{1.1}
A linear operator on associative algebra $\pi_2$ is an automorphism if and only if the matrix of this linear operator has the following matrix form
\begin{equation}  \label{(1)}
\left(
  \begin{array}{ccccc}
    a_{1,1} & 0 & 0 & 0 & 0 \\
    a_{2,1} & a_{1,1}^2 & 0 & 0 & 0 \\
    a_{3,1} & 2a_{1,1}a_{2,1} & a_{1,1}^3 & a_{3,4} & 0 \\
    a_{4,1} & 0 & 0 &   a_{1,1}+a_{4,1} & 0 \\
    a_{5,1} & 2a_{1,1}a_{4,1}+a_{4,1}^2 & 0 & a_{5,4} & (a_{1,1}+a_{4,1})^2 \\
  \end{array}
\right)
\end{equation}
\end{theorem}

\begin{theorem} \label{2.5}  \cite{1.1}
A linear operator on associative algebra $\pi_3$ is an automorphism if and only if the matrix of this linear operator has the following matrix form
\begin{equation}   \label{(2)}
\left(
  \begin{array}{ccccc}
    a_{1,1} & 0 & 0 & 0 & 0 \\
    a_{2,1} & a_{1,1}^2 & 0 & 0 & 0 \\
    a_{3,1} & 2a_{1,1}a_{2,1} & a_{1,1}^3 & a_{3,4} & 0 \\
    0 & 0 & 0 & a_{1,1} & 0 \\
    a_{5,1} & 0 & 0 & a_{5,4} & a_{1,1}^2 \\
  \end{array}
\right).
\end{equation}
\end{theorem}

\section{Local automorphisms of five-dimensional  naturally graded
associative algebra with characteristic sequence $C(A)=(3,2)$}

Now we give a description of local automorphisms of associative algebras $\pi_2$ and $\pi_3$.

\begin{definition}
Let A be an algebra. A linear map $\Phi:A\rightarrow A$ is called a local automorphism, if for any element $\nu\in A$ there exists an automorphism $\varphi_\nu:A\rightarrow A$ depending on $\nu$ such that $\Phi(\nu)=\varphi_\nu(\nu) $.
\end{definition}

\begin{theorem} \label{3.2}
A linear operator on associative algebra  $\pi_2$ is a local automorphism if and only if the matrix of this linear operator has the following matrix form
\begin{equation}  \label{pi_22}
\left(
  \begin{array}{ccccc}
    b_{11} & 0 & 0 & 0 & 0 \\
    b_{21} & b_{22} & 0 & 0 & 0 \\
    b_{31} & b_{32} & b_{33} & b_{34} & 0 \\
    b_{41} & 0 & 0 &   b_{41}+b_{11} & 0 \\
    b_{51} & b_{52} & 0 & b_{54} & b_{22}+b_{52} \\
  \end{array}
\right),
\end{equation}
where $b_{11}b_{22}b_{33}(b_{41}+b_{11})(b_{22}+b_{52})\neq 0$.
\end{theorem}

\begin{proof}
Let $\nabla$ be a local automorphism on $\pi_2$, let $B=(b_{ij} )^5_{ij=1} $ be the matrix, defining the linear operator $\nabla$. By the definition, for any element  $\nu\in\pi_2$, there exists an automorphism $L_\nu$ such that $\nabla(\nu)=L_\nu(\nu)$. Then, for the appropriate matrix $A_\nu$ of the operator $L_\nu$, we have

\begin{equation}  \label{(4)}
\overline{L_\nu(\nu)}=B\bar{\nu}=A_\nu\bar{\nu}, \ \ \ A_\nu=(a^\nu_{i,j} )^5_{i,j=1}. ~~~
\end{equation}

By the form (\ref{(1)}) of the matrix of an automorphism on $\pi_2$, using equalities $\nabla(e_i)=L_{e_i}(e_i)=B\bar{e_i},$ $i=1,2,3,4,5,$ we get
$$b_{1,1}=a_{1,1}^{e_1}\neq 0,~ b_{2,1}=a_{2,1}^{e_1},~ b_{3,1}=a_{3,1}^{e_1},~ b_{4,1}=a_{4,1}^{e_1},~ b_{5,1}=a_{5,1}^{e_1},$$
 $$b_{1,2}=0,~ b_{2,2}=(a_{1,1}^{e_2})^2\neq 0,~ b_{3,2}=2a_{1,1}^{e_2}a_{2,1}^{e_2},~ b_{4,2}=0,~ b_{5,2}=2a_{1,1}^{e_2}a_{4,1}^{e_2}+(a_{4,1}^{e_2})^2,$$
 $$b_{1,3}=0,~ b_{2,3}=0,~ b_{3,3}=(a_{1,1}^{e_3})^3,~ b_{4,3}=0,~ b_{5,3}=0,$$
 $$b_{1,4}=0,~ b_{2,4}=0,~ b_{3,4}=a_{3,4}^{e_4},~ b_{4,4}=a_{1,1}^{e_4}+a_{4,4}^{e_4}\neq 0,~ b_{5,4}=a_{5,4}^{e_4},$$
 $$b_{1,5}=0,~ b_{2,5}=0,~ b_{3,5}=0,~ b_{4,5}=0,~ b_{5,5}=(a_{1,1}^{e_5}+a_{4,4}^{e_5})^2\neq 0$$.
 Hence,
 \[
B=\left(
  \begin{array}{ccccc}
    b_{11} & 0 & 0 & 0 & 0 \\
    b_{21} & b_{22}^2 & 0 & 0 & 0 \\
    b_{31} & b_{32} & b_{33}^3 & b_{34} & 0 \\
    b_{41} & 0 & 0 &   b_{44} & 0 \\
    b_{51} & b_{52} & 0 & b_{54} & b_{55}^2 \\
  \end{array}
\right),
\]
where $b_{11}b_{22}b_{33}b_{44}b_{55}\neq 0$.
Since the components of the matrix $B$ are complex numbers we can rewrite $B$ in the following form:
\[
B=\left(
  \begin{array}{ccccc}
    b_{11} & 0 & 0 & 0 & 0 \\
    b_{21} & b_{22} & 0 & 0 & 0 \\
    b_{31} & b_{32} & b_{33} & b_{34} & 0 \\
    b_{41} & 0 & 0 &   b_{44} & 0 \\
    b_{51} & b_{52} & 0 & b_{54} & b_{55} \\
  \end{array}
\right),
\]
where $b_{11}b_{22}b_{33}b_{44}b_{55}\neq 0$.

Thus, according to (\ref{(4)}) and Theorem \ref{2.4}, we obtain the following system of linear equations
\begin{small}
\begin{equation}\label{(5)}
\left\{%
\begin{array}{lllll}
    {\nu_1\alpha_{11}^\nu=b_{11}\nu_1} \\
    {\nu_1\alpha_{21}^\nu+\nu_2(\alpha_{11}^\nu)^2=\nu_1b_{21}+\nu_2b_{22}} \\
    {\nu_1\alpha_{31}^\nu+2\nu_2\alpha_{11}^\nu\alpha_{21}^\nu+\nu_3(\alpha_{11}^\nu)^3+\nu_4\alpha_{34}^\nu=\nu_1b_{31}+\nu_2b_{32}+\nu_3b_{33}+\nu_4b_{34}} \\
    {\nu_1\alpha_{41}^\nu+\nu_4(\alpha_{11}^\nu+\alpha_{41}^\nu)=\nu_1b_{41}+\nu_4b_{44}} \\
    {\nu_1\alpha_{51}^\nu+\nu_2(2\alpha_{11}^\nu\alpha_{41}^\nu+(\alpha_{41}^\nu)^2)+\nu_4\alpha_{54}^\nu+\nu_5(\alpha_{11}^\nu+\alpha_{41}^\nu)^2=
    \nu_1b_{51}+\nu_2b_{52}+\nu_4b_{54}+\nu_5b_{55}}
\end{array}%
\right.
\end{equation}
\end{small}

If, for each element $ x\in\pi_2 $, there exists a matrix $A_x$ of the form (\ref{(1)}) such that
$$
B\bar{x}=A_x\bar{x},
$$
then the linear operator, defined by the matrix B is a local automorphism. In other words, if, for each element $x\in\pi_2$, the system of linear equations
(\ref{(5)}) has a solution with respect to the variables
$$
a_{11}^\nu, a_{21}^\nu, a_{24}^\nu, a_{31}^\nu, a_{34}^\nu, a_{41}^\nu, a_{44}^\nu, a_{51}^\nu, a_{52}^\nu, a_{54}^\nu,
$$
then the linear operator, defined by the matrix $B$ is a local automorphism. Note that, if the left part of any equation of the system (\ref{(5)}) is equal to zero, then the right part of this equation is also equal to zero. We show that for each element $x\in\pi_2 $, the system of equations (\ref{(5)}) has a solution.

1) Now, suppose that $\nu_1\neq 0.$ Then, from (\ref{(5)}) it follows that
\begin{equation}  \nonumber
\left\{%
\begin{array}{lllll}
    {\alpha_{11}^\nu=b_{11}} \\
    {\nu_1\alpha_{21}^\nu=\nu_1b_{21}+\nu_2b_{22}-\nu_2b_{11}^2}\\
    {\nu_1\alpha_{31}^\nu+2\nu_2\alpha_{11}^\nu\alpha_{21}^\nu+\nu_3b_{11}^3+\nu_4\alpha_{34}^\nu=\nu_1b_{31}+\nu_2b_{32}+\nu_3b_{33}+\nu_4b_{34}} \\
    {\nu_1\alpha_{41}^\nu+\nu_4(b_{11}+\alpha_{41}^\nu)=\nu_1b_{41}+\nu_4b_{44}} \\
    {\nu_1\alpha_{51}^\nu+\nu_2(2b_{11}\alpha_{41}^\nu+(\alpha_{41}^\nu)^2)+\nu_4\alpha_{54}^\nu+\nu_5(b_{11}+\alpha_{41}^\nu)^2
    =\nu_1b_{51}+\nu_2b_{52}+\nu_4b_{54}+\nu_5b_{55}}
\end{array}%
\right.
\end{equation}

If $0\neq \nu_1=-\nu_4$, then, from
\[
\nu_1\alpha_{41}^\nu+\nu_4b_{11}+\nu_4\alpha_{41}^\nu=\nu_1b_{41}+\nu_4b_{44}
\]
it follows that
\[
(\nu_1+\nu_4)\alpha_{41}^\nu=\nu_1b_{41}+\nu_4b_{44}-\nu_4b_{11},
\]
i.e.,
\[
\nu_1b_{41}+\nu_4b_{44}-\nu_4b_{11}=0,
\]
Hence, $b_{41}-b_{44}+b_{11}=0$, i.e., $b_{44}=b_{41}+b_{11}$.
Therefore, the system of equations (\ref{(5)}) we can rewrite in the following form:
\begin{small}
\begin{equation}\label{(6)}
\left\{%
\begin{array}{lllll}
    {\nu_1\alpha_{11}^\nu=b_{11}\nu_1} \\
    {\nu_1\alpha_{21}^\nu+\nu_2(\alpha_{11}^\nu)^2=\nu_1b_{21}+\nu_2b_{22}} \\
    {\nu_1\alpha_{31}^\nu+2\nu_2\alpha_{11}^\nu\alpha_{21}^\nu+\nu_3(\alpha_{11}^\nu)^3+\nu_4\alpha_{34}^\nu=\nu_1b_{31}+\nu_2b_{32}+\nu_3b_{33}+\nu_4b_{34}} \\
    {\nu_1\alpha_{41}^\nu+\nu_4(\alpha_{11}^\nu+\alpha_{41}^\nu)=(\nu_1+\nu_4)b_{41}+\nu_4b_{11}} \\
    {\nu_1\alpha_{51}^\nu+\nu_2(2\alpha_{11}^\nu\alpha_{41}^\nu+(\alpha_{41}^\nu)^2)+\nu_4\alpha_{54}^\nu+\nu_5(\alpha_{11}^\nu+\alpha_{41}^\nu)^2=
    \nu_1b_{51}+\nu_2b_{52}+\nu_4b_{54}+\nu_5b_{55}}
\end{array}%
\right.
\end{equation}
\end{small}
Now, if $\nu_1\neq -\nu_4$, then the system of equations (\ref{(6)}) has a solution.
Thus, we can find at least one solution of the system (\ref{(6)}) in the case $\nu_1\neq 0$.

	2) Now, suppose that $\nu_1=0 $ and $\nu_4\neq 0. $ Then we have
\begin{small}
\begin{equation}\label{(7)}
\left\{%
\begin{array}{llll}
    {\nu_2(\alpha_{11}^\nu)^2=\nu_2b_{22}} \\
    {2\nu_2\alpha_{11}^\nu\alpha_{21}^\nu+\nu_3(\alpha_{11}^\nu)^3+\nu_4\alpha_{34}^\nu=\nu_2b_{32}+\nu_3b_{33}+\nu_4b_{34}} \\
    {\nu_4(\alpha_{11}^\nu+\alpha_{41}^\nu)=\nu_4b_{41}+\nu_4b_{11}} \\
    {\nu_2(2\alpha_{11}^\nu\alpha_{41}^\nu+(\alpha_{41}^\nu)^2)+\nu_4\alpha_{54}^\nu+\nu_5(\alpha_{11}^\nu+\alpha_{41}^\nu)^2=
    \nu_2b_{52}+\nu_4b_{54}+\nu_5b_{55}}
\end{array}%
\right.
\end{equation}
\end{small}
If $\nu_2\neq 0$, then
\begin{small}
\begin{equation}\label{(8)}
\left\{%
\begin{array}{llll}
    {\alpha_{11}^\nu=\sqrt{b_{22}}} \\
    {2\nu_2\sqrt{b_{22}}\alpha_{21}^\nu+\nu_4\alpha_{34}^\nu=-\nu_3(\sqrt{b_{22}})^3+\nu_2b_{32}+\nu_3b_{33}+\nu_4b_{34}} \\
    {\alpha_{41}^\nu=-\sqrt{b_{22}}+b_{41}+b_{11}} \\
    {\nu_4\alpha_{54}^\nu=-\nu_2(2\sqrt{b_{22}}\alpha_{41}^\nu+(\alpha_{41}^\nu)^2)-\nu_5(\sqrt{b_{22}}+\alpha_{41}^\nu)^2+
    \nu_2b_{52}+\nu_4b_{54}+\nu_5b_{55}}
\end{array}%
\right.
\end{equation}
\end{small}
In the order indicated in the last system (\ref{(8)}) of equations we can find a solution of this system
if $\nu_1=0$, $\nu_2\neq 0$ and $\nu_4\neq 0$. Else, if $\nu_2=0$, then we have the following system of equations
\begin{small}
\begin{equation}\label{(9)}
\left\{%
\begin{array}{lll}
    {\nu_3(\alpha_{11}^\nu)^3+\nu_4\alpha_{34}^\nu=\nu_3b_{33}+\nu_4b_{34}} \\
    {\alpha_{11}^\nu+\alpha_{41}^\nu=b_{41}+b_{11}} \\
    {\nu_4\alpha_{54}^\nu=-\nu_5(b_{41}+b_{11})^2+\nu_4b_{54}+\nu_5b_{55}}
\end{array}%
\right.
\end{equation}
\end{small}
Since $\nu_4\neq 0$, we have the system of equations (\ref{(9)}) has a solution.

	3) If $\nu_1=0,  \nu_4=0$ and $\nu_2\neq0,$ then we have
\begin{small}
\begin{equation}\label{(10)}
\left\{%
\begin{array}{llll}
    {\alpha_{11}^\nu=\sqrt{b_{22}}} \\
    {2\nu_2\sqrt{b_{22}}\alpha_{21}^\nu=-\nu_3(\sqrt{b_{22}})^3+\nu_2b_{32}+\nu_3b_{33}} \\
    {\nu_2(2\sqrt{b_{22}}\alpha_{41}^\nu+(\alpha_{41}^\nu)^2)+\nu_5(\sqrt{b_{22}}+\alpha_{41}^\nu)^2=
    \nu_2b_{52}+\nu_5b_{55}}
\end{array}%
\right..
\end{equation}
\end{small}
This system we can rewrite in the following form
\begin{small}
\begin{equation}\label{(10)}
\left\{%
\begin{array}{llll}
    {\alpha_{11}^\nu=\sqrt{b_{22}}} \\
    {2\nu_2\sqrt{b_{22}}\alpha_{21}^\nu=-\nu_3(\sqrt{b_{22}})^3+\nu_2b_{32}+\nu_3b_{33}} \\
    ({\nu_2+\nu_5)(\sqrt{b_{22}}+\alpha_{41}^\nu)^2=\nu_2(\sqrt{b_{22}})^2+
    \nu_2b_{52}+\nu_5b_{55}}
\end{array}%
\right..
\end{equation}
\end{small}
If $\nu_2+\nu_5=0$, then $\nu_2(\sqrt{b_{22}})^2+\nu_2b_{52}-\nu_2b_{55}=0$, i.e., $b_{55}=b_{22}+b_{52}$.
Only in this case the system \ref{(10)} will have a solution. So, the system \ref{(6)} we replace by the following system:
\begin{small}
\begin{equation}\label{(11)}
\left\{%
\begin{array}{llllll}
    {\nu_1\alpha_{11}^\nu=b_{11}\nu_1} \\
    {\nu_1\alpha_{21}^\nu+\nu_2(\alpha_{11}^\nu)^2=\nu_1b_{21}+\nu_2b_{22}} \\
    {\nu_1\alpha_{31}^\nu+2\nu_2\alpha_{11}^\nu\alpha_{21}^\nu+\nu_3(\alpha_{11}^\nu)^3+\nu_4\alpha_{34}^\nu=\nu_1b_{31}+\nu_2b_{32}+\nu_3b_{33}+\nu_4b_{34}} \\
    {\nu_1\alpha_{41}^\nu+\nu_4(\alpha_{11}^\nu+\alpha_{41}^\nu)=(\nu_1+\nu_4)b_{41}+\nu_4b_{11}} \\
    {\nu_1\alpha_{51}^\nu+\nu_2(2\alpha_{11}^\nu\alpha_{41}^\nu+(\alpha_{41}^\nu)^2)+\nu_4\alpha_{54}^\nu+\nu_5(\alpha_{11}^\nu+\alpha_{41}^\nu)^2}\\
    {=\nu_1b_{51}+\nu_2b_{52}+\nu_4b_{54}+\nu_5(b_{22}+b_{52})}
\end{array}%
\right.
\end{equation}
\end{small}
Thus, the system \ref{(11)} has a solution in the case $\nu_1=0,  \nu_4=0$ and $\nu_2\neq 0$, and $\nu_2+\nu_5=0$.
Else, if $\nu_2+\nu_5\neq 0$, then the system \ref{(10)}, and, hence, the system \ref{(11)} has a solution.
So, in the case $\nu_1=0,  \nu_4=0$ and $\nu_2\neq 0$ the system \ref{(11)} has a solution.

	4) In the case $\nu_1=0,  \nu_4=0, \nu_2=0$ and $\nu_3\neq0$ we have
\begin{small}
\begin{equation}\label{(12)}
\left\{%
\begin{array}{lllll}
    {\nu_3(\alpha_{11}^\nu)^3=\nu_3b_{33}} \\
    {\nu_5(\alpha_{11}^\nu+\alpha_{41}^\nu)^2=\nu_5(b_{22}+b_{52})}
\end{array}%
\right..
\end{equation}
\end{small}
From (\ref{(12)}) it follows that
\begin{small}
\begin{equation}\label{(13)}
\left\{%
\begin{array}{lllll}
    {\alpha_{11}^\nu=\sqrt[3]{b_{33}}} \\
    {\nu_5(\sqrt[3]{b_{33}}+\alpha_{41}^\nu)^2=\nu_5(b_{22}+b_{52})}
\end{array}%
\right..
\end{equation}
\end{small}
Clearly, the system of equations \ref{(13)} has a solution for any complex number $\nu_5$.
Hence, in this case the system of equations \ref{(11)} also has a solution.

5) The case $\nu_1=0,  \nu_4=0, \nu_2=0, \nu_3=0$ and $\nu_5\in {\Bbb C}$ is obvious, i.e., in this case system of equations \ref{(11)} also has a solution.
Thus, the system of linear equations \ref{(11)} always has a solution, i.e., the linear operator, generated by the matrix (\ref{pi_22}) is a local automorphism.
The proof is complete.
\end{proof}

\begin{theorem} \label{3.3}
A linear operator on associative algebra $\pi_3 $ is a local automorphism if and only if the matrix of this linear operator has the following matrix form
\begin{equation}\label{(136)}
\left(
  \begin{array}{ccccc}
    b_{11} & 0 & 0 & 0 & 0 \\
    b_{21} & b_{11}^2 & 0 & 0 & 0 \\
    b_{31} & b_{32} & \pm b_{11}^3 & b_{34} & 0 \\
    0 & 0 & 0 &   b_{11} & 0 \\
    b_{51} & 0 & 0 & b_{54} & b_{11}^2 \\
  \end{array}
\right),
\end{equation}
where $b_{11}\neq 0$.
\end{theorem}

\begin{proof}
Let $\nabla$ be a local automorphism on $\pi_3,$ let $B=(b_{ij} )_{ij=1}^5$ be the matrix, defining the linear operator $\nabla$. By the definition, for any element $\nu\in\pi_3$, there exists a automorphism operator $L_\nu$ such that $\nabla(\nu)=L_\nu(\nu).$ Then, for the appropriate matrix $A_\nu$ of the operator $L_\nu$, we have

\begin{equation}\label{(137)}
\overline{L_\nu(\nu)}=B\overline{\nu}=A_\nu\overline{\nu}, ~~ A_\nu=(a_{ij}^\nu)_{ij=1}^5.~~
\end{equation}
By the form (2) of the matrix of an automorphism on $\pi_3,$  using equalities $\nabla(e_i)=L_{e_i}(e_i)=B(\overline{e_i}),i=1,2,3,4,5,$ we get
$$b_{1,1}=a_{1,1}^{e_1},~ b_{2,1}=a_{2,1}^{e_1},~ b_{3,1}=a_{3,1}^{e_1},~ b_{4,1}=0,~ b_{5,1}=a_{5,1}^{e_1},$$
$$
b_{1,2}=0,~ b_{2,2}=(a_{1,1}^{e_2})^2,~ b_{3,2}=2a_{1,1}^{e_2}a_{2,1}^{e_2},~ b_{4,2}=0,~ b_{5,2}=0,
$$
$$b_{1,3}=0,~ b_{2,3}=0,~ b_{3,3}=(a_{1,1}^{e_3})^3,~ b_{4,3}=0,~ b_{5,3}=0,$$
$$b_{1,4}=0,~ b_{2,4}=0,~ b_{3,4}=a_{3,4}^{e_4},~ b_{4,4}=a_{1,1}^{e_4},~ b_{5,4}=a_{5,4}^{e_4},$$
$$b_{1,5}=0,~ b_{2,5}=0,~ b_{3,5}=0,~ b_{4,5}=0,~ b_{5,5}=(a_{1,1}^{e_5})^2.$$
 Thus, according to \ref{(137)} and Theorem \ref{2.5}, we get the following system of linear equations
 \begin{equation} \label{(138)}
\left\{%
 \begin{array}{lllll}
    {\nu_1\alpha_{11}^\nu=b_{11}\nu_1} \\
    {\nu_1\alpha_{21}^\nu+\nu_2(\alpha_{11}^\nu)^2=\nu_1b_{21}+\nu_2b_{22}^2} \\
    {\nu_1\alpha_{31}^\nu+2\nu_2\alpha_{11}^\nu\alpha_{21}^\nu+\nu_3(\alpha_{11}^\nu)^3+\nu_4\alpha_{34}^\nu=\nu_1b_{31}+\nu_2b_{32}+\nu_3b_{33}^3+\nu_4b_{34}} \\
    {\nu_4\alpha_{11}^\nu=\nu_4b_{44}} \\
    {\nu_1\alpha_{51}^\nu+\nu_4\alpha_{54}^\nu+\nu_5(\alpha_{11}^\nu)^2=\nu_1b_{51}+\nu_4b_{54}+\nu_5b_{55}^2}
\end{array}%
\right.
\end{equation}

If, for each element $x\in\pi_3$, there exists a matrix $A_x$ of the form \ref{(2)} such that
$$B\overline{x}=A_x\overline{x},$$
then the linear operator, defined by the matrix B is a local automorphism. In other words, if, for each element $x\in\pi_3,$ the system of linear equations \ref{(138)} has a solution with respect to the variables
$$\alpha_{1,1}^\nu, \alpha_{2,1}^\nu, \alpha_{3,1}^\nu, \alpha_{3,4}^\nu, \alpha_{5,1}^\nu, \alpha_{5,4}^\nu,$$
then the linear operator, defined by the matrix B is a local automorphism. Note that, if the left part of any equation of the system \ref{(138)} is equal to zero, then the right part of this equation is also equal to zero. We show that for each element $x\in\pi_3,$ the system of equations \ref{(138)} has a solution.

1) Now, suppose that $\nu_1\neq0$. Then $\alpha_{11}^\nu=b_{11}.$ From fourth equation of the system \ref{(138)} it follows that  $b_{11}=b_{44}$ if $\nu_4\neq0.$ So, we should take $b_{11}=b_{44}$ in the system \ref{(138)}. Since the variables of the other equations of the system \ref{(138)} at the coefficient $\nu_1$ are pairwise distinct, in this case the system \ref{(138)} has a solution for any $\nu_1\neq0, \nu_2,  \nu_3,  \nu_4 $ and $\nu_5$ if $b_{11}=b_{44}$.

	2) In the case $\nu_1=0$ and $\nu_2\neq0$ from \ref{(138)} it follows that
\begin{equation}   \label{(139)}
\left\{%
 \begin{array}{lllll}
    {0=0} \\
    {\nu_2(\alpha_{11}^\nu)^2=\nu_2b_{22}^2} \\
    {2\nu_2\alpha_{11}^\nu\alpha_{21}^\nu+\nu_3(\alpha_{11}^\nu)^3+\nu_4\alpha_{34}^\nu=\nu_2b_{32}+\nu_3b_{33}^3+\nu_4b_{34}} \\
    {\nu_4\alpha_{11}^\nu=\nu_4b_{11}} \\
    {\nu_4\alpha_{54}^\nu+\nu_5(\alpha_{11}^\nu)^2=\nu_4b_{54}+\nu_5b_{55}^2}

\end{array}%
\right..
\end{equation}
Then $(\alpha_{11}^\nu)^2=b_{22}^2.$ Hence $\alpha_{11}^\nu=\mp b_{22}.$ From the fourth equation of the system \ref{(138)} it follows that  $b_{11}=\mp b_{22}$ if $ \nu_4\neq0.$ So, we should take $b_{22}^2=b_{11}^2$ in the system \ref{(138)}.
Also, if $\nu_4=0$ and $\nu_5\neq 0$, then we should take $b_{55}^2=b_{11}^2$.

Hence, from \ref{(139)} it follows that
\begin{equation}\label{(1310)}
\left\{%
 \begin{array}{lllll}
    {0=0} \\
    {\nu_2(\alpha_{11}^\nu)^2=\nu_2b_{11}^2} \\
    {2\nu_2b_{11}\alpha_{21}^\nu+\nu_3b_{11}^3+\nu_4a_{34}^\nu=\nu_2b_{32}+\nu_3b_{33}^3+\nu_4b_{34}} \\
    {\nu_4b_{11}=\nu_4b_{11}} \\
    {\nu_4\alpha_{54}^\nu+\nu_5b_{11}^2=\nu_4b_{54}+\nu_5b_{11}^2}

\end{array}%
\right..
\end{equation}
Clearly, the system \ref{(1310)} of linear equations has a solution with respect to the variables $\alpha_{21}^\nu, \alpha_{34}^\nu$ and $\alpha_{54}^\nu.$ Therefore, in this case, the system \ref{(138)} also has a solution for any $\nu_1=0,  \nu_2\neq0, \nu_3, \nu_4$ and $\nu_5$ if $b_{22}^2=b_{55}^2=b_{11}^2$.

	3) If $\nu_1=0$, $\nu_2=0$ and $\nu_4\neq 0$, then
\begin{equation}\label{(1311)}
\left\{%
 \begin{array}{lllll}
    {0=0} \\
    {0=0} \\
    {\nu_3(\alpha_{11}^\nu)^3+\nu_4\alpha_{34}^\nu=\nu_3b_{33}^3+\nu_4b_{34}} \\
    {\nu_4\alpha_{11}^\nu=\nu_4b_{11}} \\
    {\nu_4\alpha_{54}^\nu+\nu_5(\alpha_{11}^\nu)^2=\nu_4b_{54}+\nu_5b_{11}^2}
\end{array}%
\right..
\end{equation}
Hence, $\alpha_{11}^\nu=b_{11}$ and from \ref{(1311)} it follows that
\begin{equation}\label{(1312)}
\left\{%
 \begin{array}{lllll}
    {0=0} \\
    {0=0} \\
    {\nu_3b_{11}^3+\nu_4\alpha_{34}^\nu=\nu_3b_{33}^3+\nu_4b_{34}} \\
    {\nu_4b_{11}=\nu_4b_{11}} \\
    {\nu_4\alpha_{54}^\nu+\nu_5b_{11}^2=\nu_4b_{54}+\nu_5b_{11}^2}
\end{array}%
\right..
\end{equation}
Clearly, the system \ref{(1312)} of linear equations has a solution with respect to the variables $\alpha_{34}^\nu$ and $\alpha_{54}^\nu$.

	4) If $\nu_1=0, \nu_2=0, \nu_4=0 $ and $\nu_3\neq 0,$ then
\begin{equation}\label{(1313)}
\left\{%
 \begin{array}{lllll}
    {0=0} \\
    {0=0} \\
    {\nu_3(\alpha_{11}^\nu)^3=\nu_3b_{33}^3} \\
    {0=0} \\
    {\nu_5(\alpha_{11}^\nu)^2=\nu_5b_{11}^2}
\end{array}%
\right..
\end{equation}
Then $\alpha_{11}^\nu=b_{33}$. From fifth equation of the system \ref{(1313)} it follows that  $(\alpha_{11}^\nu)^2=b_{11}^2$, i.e., $\alpha_{11}^\nu=\pm b_{11}$, if $\nu_5\neq 0$. So, we should take $b_{33}^3=\pm b_{11}^3$ in the system \ref{(138)}. In the case $\nu_1=0, \nu_2=0$ and $\nu_4=0$ the system (\ref{(1313)}) has a solution for any $\nu_3\neq 0$ and $\nu_5$ if $b_{33}^3=\pm b_{11}^3$. So, the system \ref{(138)} also has a solution for any $\nu_1=0$,  $\nu_2=0$, $\nu_4=0$, $\nu_3\neq 0$ and $\nu_5$ if $b_{22}^2=b_{55}^2=b_{11}^2$ and $b_{33}^3=\pm b_{11}^3$.

		5) In the case $\nu_1=0, \nu_2=0, \nu_3=0, \nu_4=0$ the system \ref{(138)} also has a solution. Hence, the system \ref{(138)} has a solution for any $\nu_1$,  $\nu_2$, $\nu_3$, $\nu_4$ and $\nu_5$ if $b_{22}^2=b_{55}^2=b_{11}^2$ and $b_{33}^3=\pm b_{11}^3$.

Thus, the system of linear equations
\begin{equation}\label{eq}\nonumber
\left\{%
 \begin{array}{lllll}
    {\nu_1\alpha_{11}^\nu=b_{11}\nu_1} \\
    {\nu_1\alpha_{21}^\nu+\nu_2(\alpha_{11}^\nu)^2=\nu_1b_{21}+\nu_2b_{11}^2} \\
    {\nu_1\alpha_{31}^\nu+2\nu_2\alpha_{11}^\nu\alpha_{21}^\nu+\nu_3(\alpha_{11}^\nu)^3+\nu_4\alpha_{34}^\nu=\nu_1b_{31}+\nu_2b_{32}\pm\nu_3b_{11}^3+\nu_4b_{34}} \\
    {\nu_4\alpha_{11}^\nu=\nu_4b_{44}} \\
    {\nu_1\alpha_{51}^\nu+\nu_4\alpha_{54}^\nu+\nu_5(\alpha_{11}^\nu)^2=\nu_1b_{51}+\nu_4b_{54}+\nu_5b_{11}^2}

\end{array}%
\right.
\end{equation}
always has a solution, i.e., the linear operator, generated by the matrix (\ref{(136)}) is a local automorphism.
The proof is complete.
\end{proof}

\section{Principles for derivations and local derivations of finite dimensional algebras}

Let $\mathcal{A}$ be an algebra, and, let $D:\mathcal{A}\to \mathcal{A}$ be a linear mapping. Then $D$ is called derivation if,
for any $x$, $y\in \mathcal{A}$, $D(xy)=D(x)y+xD(y)$.

A linear mapping $\nabla:\mathcal{A}\to \mathcal{A}$ is called local derivation if, for each $x\in\mathcal{A}$, there exists a derivation
$D_x:\mathcal{A}\to \mathcal{A}$ depending on $x$ such that $\nabla(x)=D_x(x)$.
Then the following theorem takes place.

\begin{theorem} \label{4.1}
Let $\mathcal{A}$ be an $n$-dimensional algebra, and, let $D:\mathcal{A}\to \mathcal{A}$ be a derivation and
$\nabla:\mathcal{A}\to \mathcal{A}$ be a local derivation of $\mathcal{A}$. Let $M=(a_{i,j})_{i,j=1}^n$ be the matrix form
of derivations and $N=(b_{i,j})_{i,j=1}^n$ be the matrix form of local derivations of the algebra $\mathcal{A}$.
Then the following principles takes place

0) for any $i,j$ from $\{1,2,\dots,n\}$, if $a_{i,j}=0$, then $b_{i,j}=0$.

1) for any $i,j,k,m$ from $\{1,2,\dots,n\}$ such that $i<k$, $j<m$, if $a_{i,j}=a_{k,m}$, $a_{i,m}=a_{k,j}=0$,
then $b_{i,j}=b_{k,m}$.

2) for any $i,j,k,m$ from $\{1,2,\dots,n\}$ such that $i<k$, $j<m$, if $a_{i,j}=a_{k,m}$ and one
of the components $a_{i,m}$, $a_{k,j}$ is nonzero and is not contained in the expression of all other components of the matrix form $M$,
then the components $b_{i,j}$ $b_{k,m}$ are generally distinct.

3) for any $i,j,k,m$ from $\{1,2,\dots,n\}$ such that $i<k$, $j<m$, if one
of the components $a_{i,m}$, $a_{k,j}$ is nonzero and is not contained in the expression of all other components of the matrix form $M$,
then the components $b_{i,j}$ $b_{k,m}$ are generally distinct.

4) for any $i$ from $\{1,2,\dots,n\}$, if the component $a_{i+2,i}$ is nonzero and is not contained in the expression of all other components of the matrix form $M$ and $a_{i+1,i}\neq 0$, $a_{i+2,i+1}\neq 0$, $a_{i,i}\neq 0$, $a_{i+1,i+1}\neq 0$ and $a_{i+2,i+2}\neq 0$ then the components $b_{i+1,i}$, $b_{i+2,i+1}$, $b_{i,i}$, $b_{i+1,i+1}$ and $b_{i+2,i+2}$ are generally pairwise distinct.
\end{theorem}

\begin{proof}
Let $\nabla$ be a local derivation on $\mathcal{A}$, let $B=(b_{ij} )^n_{ij=1} $ be the matrix of the form $N$, defining the local derivation $\nabla$. By the definition, for any element  $\nu\in\mathcal{A}$, there exists a derivation $D_\nu$ such that $\nabla(\nu)=D_\nu(\nu)$. Then, for the appropriate matrix $A_\nu$ of the form $M$ of the operator $D_\nu$, we have

\begin{equation}  \label{(30)}
\overline{D_\nu(\nu)}=B\bar{\nu}=A_\nu\bar{\nu}, \ \ \ A_\nu=(a^\nu_{i,j} )^n_{i,j=1}. ~~~
\end{equation}

By the form $M$ of the matrix of a automorphism on $\mathcal{A}$, using equalities $\nabla(e_i)=D_{e_i}(e_i)=B\bar{e_i},$ $i=1,2,\dots,n$ we get
\[
b_{\xi,\eta}=a_{\xi,\eta}^{e_\eta}, \xi,\eta=1,2,\dots,n.
\]
In particular,
\[
b_{i,j}=a_{i,j}^{e_j}, b_{k,m}=a_{k,m}^{e_m}, b_{i,m}=a_{i,m}^{e_m}, b_{k,j}=a_{k,j}^{e_j}.
\]

First we prove Item 1) of Theorem \ref{4.1}.
Since $\nabla$ is additive, we have
\[
\nabla(e_j+e_m)=\nabla(e_j)+\nabla(e_m).
\]
By the condition of Item 1) of Theorem \ref{4.1}, from this it follows that
\[
a_{i,j}^{e_j+e_m}=a_{i,j}^{e_j}, a_{k,m}^{e_j+e_m}=a_{k,m}^{e_m},
\]
\[
a_{i,j}^{e_j+e_m}=a_{k,m}^{e_j+e_m}.
\]
So, $a_{i,j}^{e_j}=a_{k,m}^{e_m}$. Hence, $b_{i,j}=b_{k,m}$.
This ends the proof of Item 1).

Now we prove Item 2) of Theorem \ref{4.1}. Suppose that $a_{i,m}\neq 0$, in particular, we may suppose that
$a_{i,m}^\nu\neq 0$ for any element  $\nu\in\mathcal{A}$.
The equality $A_\nu\bar{\nu}=B\bar{\nu}$ from \ref{(30)} gives
a system of linear equations
\begin{equation}\label{(31)}
\left\{%
\begin{array}{lllll}
    {a^\nu_{1,1}\nu_1+a^\nu_{1,2}\nu_2+\dots+a^\nu_{1,n}\nu_n=b_{1,1}\nu_1+b_{1,2}\nu_2+\dots+b_{1,n}\nu_n} \\
    {a^\nu_{2,1}\nu_1+a^\nu_{2,2}\nu_2+\dots+a^\nu_{2,n}\nu_n=b_{2,1}\nu_1+b_{2,2}\nu_2+\dots+b_{2,n}\nu_n} \\
    {\dots \dots \dots \dots \dots \dots \dots \dots \dots\dots \dots \dots \dots \dots \dots \dots \dots \dots \dots} \\
    {a^\nu_{n,1}\nu_1+a^\nu_{n,2}\nu_2+\dots+a^\nu_{n,n}\nu_n=b_{n,1}\nu_1+b_{n,2}\nu_2+\dots+b_{n,n}\nu_n} \\
\end{array}%
\right.
\end{equation}
with respect to the variables
$$
a^\nu_{i,j}, i=1,2,\dots,n, j=1,2,\dots,n.
$$
Since the matrix $B$ is of the form $N$ and the matrix $A_\nu$ is of the form $M$, we have, for any $\nu\in\mathcal{A}$,
the system of linear equations \ref{(31)} has a solution. At the same time,
the system of linear equations
\begin{equation}\label{(32)}
\left\{%
\begin{array}{lllll}
    {a^\nu_{i,j}\nu_j+a^\nu_{i,m}\nu_m=b_{i,j}\nu_j+b_{i,m}\nu_m} \\
    {a^\nu_{k,j}\nu_j+a^\nu_{k,m}\nu_m=b_{k,j}\nu_j+b_{k,m}\nu_m} \\
\end{array}%
\right.
\end{equation}
with respect to the variables
$$
a^\nu_{i,j}, a^\nu_{i,m}, a^\nu_{k,j}, a^\nu_{k,m}
$$
has a solution for any $\nu\in\mathcal{A}$ if even $a^\nu_{i,j}=a^\nu_{k,m}$ and $b_{i,j}\neq b_{k,m}$.
Therefore, by the condition of Item 2) of Theorem \ref{4.1}, the system of linear equations
\ref{(31)} has a solution for any $\nu\in\mathcal{A}$ if even $a^\nu_{i,j}=a^\nu_{k,m}$ and $b_{i,j}\neq b_{k,m}$.

Items 3) and 4) of Theorem \ref{4.1} are proved similar to the proof of  Item 2). The proof is complete.
\end{proof}

\section{Local derivations of five-dimensional  naturally graded
associative algebra with characteristic sequence $C(A)=(3,2)$}

Now we give a description of derivations and local derivations of the five-dimensional  naturally graded
associative algebras $\pi_2$ and $\pi_3$ with characteristic sequence $C(\mathcal{A})=(3,2)$.

We obtain the following theorem by the definition of a derivation.

\begin{theorem} \label{Der_pi_2}
A linear operator on algebra $\pi_2$ is a derivation if and only if its matrix has the following form
\begin{equation}\label{Der_pi_2_matrix}
\left(
  \begin{array}{ccccc}
    a_{1,1} & 0 & 0 & 0 & 0 \\
    a_{2,1} & 2a_{1,1} & 0 & 0 & 0 \\
    a_{3,1} & 2a_{2,1} & 3a_{1,1} & a_{3,4} & 0 \\
    a_{4,1} & 0 & 0 & a_{4,1}+a_{1,1} & 0 \\
    a_{5,1} & 2a_{4,1}  & 0 & a_{5,4} & 2(a_{4,1}+a_{1,1}) \\
  \end{array}
\right)
\end{equation}
\end{theorem}

\begin{theorem} \label{LocDer_pi_2}
A linear operator on algebra $\pi_2$ is a local derivation if and only if its matrix has the following form
\[
\left(
  \begin{array}{ccccc}
    b_{11} & 0 & 0 & 0 & 0 \\
    b_{21} & b_{22} & 0 & 0 & 0 \\
    b_{31} & b_{32} & b_{33} & b_{34} & 0 \\
    b_{41} & 0 & 0 &   b_{41}+b_{11} & 0 \\
    b_{51} & b_{52} & 0 & b_{54} & b_{22}+b_{52} \\
  \end{array}
\right).
\]
\end{theorem}

\begin{proof}
Let $\nabla$ be a local derivation of $\pi_2$, and, let $B$ be the matrix of $\nabla$.
We will prove this theorem applying Theorem \ref{4.1}.
By Item 0) of Theorem \ref{4.1}, we have
\[
B=
\left(
  \begin{array}{ccccc}
    b_{1,1} & 0 & 0 & 0 & 0 \\
    b_{2,1} & b_{2,2} & 0 & 0 & 0 \\
    b_{3,1} & b_{3,2} & b_{3,3} & b_{3,4} & 0 \\
    b_{4,1} & 0 & 0 & b_{4,4} & 0 \\
    b_{5,1} & b_{5,2}  & 0 & b_{5,4} & b_{5,5} \\
  \end{array}
\right).
\]
Now, by Item 4) of Theorem \ref{4.1}, we have the elements
$b_{1,1}$, $b_{2,1}$, $b_{2,2}$, $b_{3,1}$, $b_{3,2}$ and $b_{3,3}$ are pairwise distinct. By Item 3) of Theorem \ref{4.1},
$b_{3,3}$, $b_{4,4}$, $b_{5,5}$ are pairwise distinct.
Similarly, $b_{4,1}$ and $b_{5,2}$ are mutually distinct. As a result, we have
the following system of linear equations
\begin{equation}\label{(5311)}
\left\{%
\begin{array}{lllll}
    {a^\nu_{1,1}\nu_1=b_{1,1}\nu_1} \\
    {a^\nu_{2,1}\nu_1+2a^\nu_{1,1}\nu_2=b_{2,1}\nu_1+b_{2,2}\nu_2} \\
    {a^\nu_{3,1}\nu_1+2a^\nu_{2,1}\nu_2+3a^\nu_{1,1}\nu_3+a^\nu_{3,4}\nu_4=b_{3,1}\nu_1+b_{3,2}\nu_2+b_{3,3}\nu_1+b_{3,4}\nu_4} \\
    {a^\nu_{4,1}\nu_1+(a^\nu_{4,1}+a^\nu_{1,1})\nu_4=b_{4,1}\nu_1+b_{4,4}\nu_4} \\
    {a^\nu_{5,1}\nu_1+2a^\nu_{4,1}\nu_2+a^\nu_{5,4}\nu_4+2(a^\nu_{4,1}+a^\nu_{1,1})\nu_5=b_{5,1}\nu_1+b_{5,2}\nu_2+b_{5,4}\nu_4+b_{5,5}\nu_5} \\
\end{array}%
\right.
\end{equation}

Similar to the proof of Theorem 7, we have the following suggestions:
if, for each element $ x\in\pi_2 $, there exists a matrix $A_x$ of the form (\ref{Der_pi_2_matrix}) such that
$$
B\bar{x}=A_x\bar{x},
$$
then the linear operator, defined by the matrix B is a local derivation. In other words, if, for each element $x\in\pi_2$, the system of linear equations
(\ref{(5)}) has a solution with respect to the variables
$$
a_{11}^\nu, a_{21}^\nu, a_{24}^\nu, a_{31}^\nu, a_{34}^\nu, a_{41}^\nu, a_{44}^\nu, a_{51}^\nu, a_{52}^\nu, a_{54}^\nu,
$$
then the linear operator, defined by the matrix $B$ is a local derivation. Note that, if the left part of any equation of the system (\ref{(5)}) is equal to zero, then the right part of this equation is also equal to zero.


1) Now, suppose that $\nu_1\neq 0.$ Then, from (\ref{(5)}) it follows that
\begin{equation}  \nonumber
\left\{%
\begin{array}{lllll}
    {\alpha_{11}^\nu=b_{11}} \\
    {\nu_1\alpha_{21}^\nu=\nu_1b_{21}+\nu_2b_{22}-\nu_22b_{11}}\\
    {\nu_1\alpha_{31}^\nu+2\nu_2\alpha_{21}^\nu+\nu_33b_{11}+\nu_4\alpha_{34}^\nu=\nu_1b_{31}+\nu_2b_{32}+\nu_3b_{33}+\nu_4b_{34}} \\
    {\nu_1\alpha_{41}^\nu+\nu_4(b_{11}+\alpha_{41}^\nu)=\nu_1b_{41}+\nu_4b_{44}} \\
    {\nu_1\alpha_{51}^\nu+\nu_2 2\alpha_{41}^\nu+\nu_4\alpha_{54}^\nu+\nu_52(b_{11}+\alpha_{41}^\nu)
    =\nu_1b_{51}+\nu_2b_{52}+\nu_4b_{54}+\nu_5b_{55}}
\end{array}%
\right.
\end{equation}

If $0\neq \nu_1=-\nu_4$, then, from
\[
\nu_1\alpha_{41}^\nu+\nu_4b_{11}+\nu_4\alpha_{41}^\nu=\nu_1b_{41}+\nu_4b_{44}
\]
it follows that
\[
(\nu_1+\nu_4)\alpha_{41}^\nu=\nu_1b_{41}+\nu_4b_{44}-\nu_4b_{11},
\]
i.e.,
\[
\nu_1b_{41}+\nu_4b_{44}-\nu_4b_{11}=0,
\]
Hence, $b_{41}-b_{44}+b_{11}=0$, i.e., $b_{44}=b_{41}+b_{11}$.
Therefore, the system of equations (\ref{(5311)}) we can rewrite in the following form:
\begin{small}
\begin{equation}\label{(56)}
\left\{%
\begin{array}{lllll}
    {a^\nu_{1,1}\nu_1=b_{1,1}\nu_1} \\
    {a^\nu_{2,1}\nu_1+2a^\nu_{1,1}\nu_2=b_{2,1}\nu_1+b_{2,2}\nu_2} \\
    {a^\nu_{3,1}\nu_1+2a^\nu_{2,1}\nu_2+3a^\nu_{1,1}\nu_3+a^\nu_{3,4}\nu_4=b_{3,1}\nu_1+b_{3,2}\nu_2+b_{3,3}\nu_3+b_{3,4}\nu_4} \\
    {a^\nu_{4,1}\nu_1+(a^\nu_{4,1}+a^\nu_{1,1})\nu_4=b_{4,1}\nu_1+(b_{41}+b_{11})\nu_4} \\
    {a^\nu_{5,1}\nu_1+2a^\nu_{4,1}\nu_2+a^\nu_{5,4}\nu_4+2(a^\nu_{4,1}+a^\nu_{1,1})\nu_5=b_{5,1}\nu_1+b_{5,2}\nu_2+b_{5,4}\nu_4+b_{5,5}\nu_5} \\
\end{array}%
\right.
\end{equation}
\end{small}
Now, if $\nu_1\neq -\nu_4$, then the system of equations (\ref{(56)}) has a solution.
Thus, we can find at least one solution of the system (\ref{(56)}) in the case $\nu_1\neq 0$.

	2) Now, suppose that $\nu_1=0 $ and $\nu_4\neq 0. $ Then we have
\begin{small}
\begin{equation}\label{(57)}
\left\{%
\begin{array}{llll}
    {2a^\nu_{1,1}\nu_2=b_{2,2}\nu_2} \\
    {2a^\nu_{2,1}\nu_2+3a^\nu_{1,1}\nu_3+a^\nu_{3,4}\nu_4=b_{3,2}\nu_2+b_{3,3}\nu_3+b_{3,4}\nu_4} \\
    {(a^\nu_{4,1}+a^\nu_{1,1})\nu_4=(b_{41}+b_{11})\nu_4} \\
    {2a^\nu_{4,1}\nu_2+a^\nu_{5,4}\nu_4+2(a^\nu_{4,1}+a^\nu_{1,1})\nu_5=b_{5,2}\nu_2+b_{5,4}\nu_4+b_{5,5}\nu_5} \\
\end{array}%
\right.
\end{equation}
\end{small}
If $\nu_2\neq 0$, then
\begin{small}
\begin{equation}\label{(58)}
\left\{%
\begin{array}{llll}
    {a^\nu_{1,1}=\frac{1}{2}b_{2,2}} \\
    {2a^\nu_{2,1}\nu_2+a^\nu_{3,4}\nu_4=-\frac{3}{2}b_{2,2}\nu_3+b_{3,2}\nu_2+b_{3,3}\nu_3+b_{3,4}\nu_4} \\
    {a^\nu_{4,1}=-\frac{3}{2}b_{2,2}+b_{41}+b_{11}} \\
    {a^\nu_{5,4}\nu_4=-2a^\nu_{4,1}\nu_2-2(a^\nu_{4,1}+a^\nu_{1,1})\nu_5+b_{5,2}\nu_2+b_{5,4}\nu_4+b_{5,5}\nu_5} \\
\end{array}%
\right.
\end{equation}
\end{small}
In the order indicated in the last system (\ref{(58)}) of equations we can find a solution of this system
if $\nu_1=0$, $\nu_2\neq 0$ and $\nu_4\neq 0$. Else, if $\nu_2=0$, then we have the following system of equations
\begin{small}
\begin{equation}\label{(59)}
\left\{%
\begin{array}{lll}
    {3a^\nu_{1,1}\nu_3+a^\nu_{3,4}\nu_4=b_{3,3}\nu_3+b_{3,4}\nu_4} \\
    {a^\nu_{4,1}+a^\nu_{1,1}=b_{41}+b_{11}} \\
    {a^\nu_{5,4}\nu_4=-2(b_{41}+b_{11})\nu_5+b_{5,4}\nu_4+b_{5,5}\nu_5} \\
\end{array}%
\right.
\end{equation}
\end{small}
Since $\nu_4\neq 0$, we have the system of equations (\ref{(59)}) has a solution.

	3) If $\nu_1=0,  \nu_4=0$ and $\nu_2\neq0,$ then we have
\begin{small}
\begin{equation}\label{(510)}
\left\{%
\begin{array}{lll}
    {a^\nu_{1,1}=\frac{1}{2}b_{2,2}} \\
    {2a^\nu_{2,1}\nu_2=-\frac{3}{2}b_{2,2}\nu_3+b_{3,2}\nu_2+b_{3,3}\nu_3} \\
    {2a^\nu_{4,1}\nu_2+2(a^\nu_{4,1}+\frac{1}{2}b_{2,2})\nu_5=b_{5,2}\nu_2+b_{5,5}\nu_5} \\
\end{array}%
\right.
\end{equation}
\end{small}
This system we can rewrite in the following form
\begin{small}
\begin{equation}\label{(5101)}
\left\{%
\begin{array}{llll}
    {a^\nu_{1,1}=\frac{1}{2}b_{2,2}} \\
    {2a^\nu_{2,1}\nu_2=-\frac{3}{2}b_{2,2}\nu_3+b_{3,2}\nu_2+b_{3,3}\nu_3} \\
    {2(a^\nu_{4,1}+\frac{1}{2}b_{2,2})(\nu_2+\nu_5)=b_{2,2}\nu_2+b_{5,2}\nu_2+b_{5,5}\nu_5} \\
\end{array}%
\right.
\end{equation}
\end{small}
If $\nu_2+\nu_5=0$, then $b_{2,2}\nu_2+b_{5,2}\nu_2-b_{5,5}\nu_2=0$, i.e., $b_{55}=b_{22}+b_{52}$.
Only in this case the system \ref{(510)} will have a solution. So, the system \ref{(56)} we replace by the following system:
\begin{small}
\begin{equation}\label{(511)}
\left\{%
\begin{array}{llllll}
    {a^\nu_{1,1}\nu_1=b_{1,1}\nu_1} \\
    {a^\nu_{2,1}\nu_1+2a^\nu_{1,1}\nu_2=b_{2,1}\nu_1+b_{2,2}\nu_2} \\
    {a^\nu_{3,1}\nu_1+2a^\nu_{2,1}\nu_2+3a^\nu_{1,1}\nu_3+a^\nu_{3,4}\nu_4=b_{3,1}\nu_1+b_{3,2}\nu_2+b_{3,3}\nu_3+b_{3,4}\nu_4} \\
    {a^\nu_{4,1}\nu_1+(a^\nu_{4,1}+a^\nu_{1,1})\nu_4=b_{4,1}\nu_1+(b_{41}+b_{11})\nu_4} \\
    {a^\nu_{5,1}\nu_1+2a^\nu_{4,1}\nu_2+a^\nu_{5,4}\nu_4+2(a^\nu_{4,1}+a^\nu_{1,1})\nu_5=b_{5,1}\nu_1+b_{5,2}\nu_2+b_{5,4}\nu_4+(b_{22}+b_{52})\nu_5} \\
\end{array}%
\right.
\end{equation}
\end{small}
Thus, the system \ref{(511)} has a solution in the case $\nu_1=0,  \nu_4=0$ and $\nu_2\neq 0$, and $\nu_2+\nu_5=0$.
Else, if $\nu_2+\nu_5\neq 0$, then the system \ref{(510)}, and, hence, the system \ref{(511)} has a solution.
So, in the case $\nu_1=0,  \nu_4=0$ and $\nu_2\neq 0$ the system \ref{(511)} has a solution.

	4) In the case $\nu_1=0,  \nu_4=0, \nu_2=0$ and $\nu_3\neq0$ we have
\begin{small}
\begin{equation}\label{(512)}
\left\{%
\begin{array}{ll}
    {3a^\nu_{1,1}\nu_3=b_{3,3}\nu_3} \\
    {2(a^\nu_{4,1}+a^\nu_{1,1})\nu_5=b_{5,5}\nu_5} \\
\end{array}%
\right..
\end{equation}
\end{small}
From (\ref{(512)}) it follows that
\begin{small}
\begin{equation}\label{(513)}
\left\{%
\begin{array}{lllll}
    {a^\nu_{1,1}=\frac{1}{3}b_{3,3}} \\
    {2(a^\nu_{4,1}+\frac{1}{3}b_{3,3})\nu_5=b_{5,5}\nu_5} \\
\end{array}%
\right..
\end{equation}
\end{small}
Clearly, the system of equations \ref{(513)} has a solution for any complex number $\nu_5$.
Hence, in this case the system of equations \ref{(511)} also has a solution.

5) The existence of a solution of \ref{(511)} in the case $\nu_1=0,  \nu_4=0, \nu_2=0, \nu_3=0$ and $\nu_5\in {\Bbb C}$ is obvious.
Thus, the system of linear equations \ref{(511)} always has a solution, i.e., the linear operator, generated by the matrix (\ref{pi_22}) is a local derivation.
The proof is complete.
\end{proof}

The following theorem is also obtained by the definition of a derivation.

\begin{theorem} \label{Der_pi_3}
A linear operator on algebra $\pi_3$ is a derivation if and only if its matrix has the following form
\[
\left(
  \begin{array}{ccccc}
    a_{1,1} & 0 & 0 & 0 & 0 \\
    a_{2,1} & 2a_{1,1} & 0 & 0 & 0 \\
    a_{3,1} & 2a_{2,1} & 3a_{1,1} & a_{3,4} & 0 \\
    0 & 0 & 0 & a_{1,1} & 0 \\
    a_{5,1} & 0  & 0 & a_{5,4} & 2a_{1,1} \\
  \end{array}
\right)
\]
\end{theorem}

\begin{theorem} \label{LocDer_pi_3}
A linear operator on algebra $\pi_3$ is a local derivation if and only if its matrix has the following form
\[
\left(
  \begin{array}{ccccc}
 b_{1,1} & 0 & 0 & 0 & 0 \\
    b_{2,1} & 2b_{1,1} & 0 & 0 & 0 \\
    b_{3,1} & b_{3,2} & 3b_{1,1} & b_{3,4} & 0 \\
    0 & 0 & 0 & b_{1,1} & 0 \\
    b_{5,1} & 0  & 0 & b_{5,4} & 2b_{1,1} \\
  \end{array}
\right)
\]
\end{theorem}

\begin{proof}
Similar to the proof of Theorem \ref{LocDer_pi_2} we can prove this theorem applying Theorem \ref{4.1}.
\end{proof}

\section{The local derivations of $\pi_2$ and $\pi_2$ form a Lie algebra}

\begin{theorem} \label{6.2}
The vector space $LocDer(\pi_3)$ of all local derivations of the algebra $\pi_3$ forms a Lie algebra with respect
to the Lie multiplication $[\nabla,\Delta]=\nabla\Delta-\Delta\nabla$.
\end{theorem}

\begin{proof}
Let $\nabla$, $\Delta$ be local derivations of algebra $\pi_3$, and, let $A$, $B$ their matrices respectively and
\[
A=
\left(
  \begin{array}{ccccc}
    x_{1} & 0 & 0 & 0 & 0 \\
    x_{4} & 2x_{1} & 0 & 0 & 0 \\
    x_{6} & x_{5} & 3x_{1} & x_{2} & 0 \\
    0 & 0 & 0 &   x_{1} & 0 \\
    x_{7} & 0 & 0 & x_{3} & 2x_{1} \\
  \end{array}
\right),
B=
\left(
  \begin{array}{ccccc}
    y_{1} & 0 & 0 & 0 & 0 \\
    y_{4} & 2y_{1} & 0 & 0 & 0 \\
    y_{6} & y_{5} & 3y_{1} & y_{2} & 0 \\
    0 & 0 & 0 &   y_{1} & 0 \\
    y_{7} & 0 & 0 & y_{3} & 2y_{1} \\
  \end{array}
\right).
\]
We compute the Lie product of $A$ and $B$.
\[
[A,B]=AB-BA=
\]
\[
\left(
  \begin{array}{ccccc}
    0 & 0 & 0 & 0 & 0 \\
    x_{1}y_{4}-x_{4}y_{1} & 0 & 0 & 0 & 0 \\
    2x_{1}y_{6}+x_{5}y_{4}-x_{4}y_{5}-2x_{6}y_{1} & x_{1}y_{5}-x_{5}y_{1} & 0 & 2x_{1}y_{2}-2x_{2}y_{1} & 0 \\
    0 & 0 & 0 &   0 & 0 \\
    x_{1}y_{7}-x_{7}y_{1} & 0 & 0 & x_{1}y_{3}-x_{3}y_{1} & 0 \\
  \end{array}
\right).
\]
Clearly, $[A,B]$ is the matrix of $[\nabla,\Delta]$ and is of the matrix form \ref{(2)}.
Hence, $[\nabla,\Delta]$ is a local derivation. Since  $\nabla$, $\Delta$  are arbitrarily chosen, we have
the vector space $LocDer(\pi_3)$ is a Lie algebra with respect to the Lie multiplication $[\nabla,\Delta]=\nabla\Delta-\Delta\nabla$.
\end{proof}

\begin{theorem}  \label{6.1}
The vector space $LocDer(\pi_2)$ of all local derivations of the algebra $\pi_2$ forms a Lie algebra with respect
to the Lie multiplication $[\nabla,\Delta]=\nabla\Delta-\Delta\nabla$.
\end{theorem}

Theorem \ref{6.1} is proved similar to the proof of the following theorem.

\section{The relationship between local derivations and local automorphisms of $\pi_2$ and $\pi_2$}

\begin{proposition} \cite[Lemma 4]{EZLS}
Local automorphisms of an arbitrary finite-dimensional algebra $\mathcal{A}$ form a group under multiplication.
\end{proposition}

Let $LocAut_+(\pi_3)$ be the subgroup of the group $LocAut(\pi_3)$ of all local automorphisms of algebra $\pi_3$ with a matrix of the form
\begin{equation}\label{(137)}
\left(
  \begin{array}{ccccc}
    b_{11} & 0 & 0 & 0 & 0 \\
    b_{21} & b_{11}^2 & 0 & 0 & 0 \\
    b_{31} & b_{32} & b_{11}^3 & b_{34} & 0 \\
    0 & 0 & 0 &   b_{11} & 0 \\
    b_{51} & 0 & 0 & b_{54} & b_{11}^2 \\
  \end{array}
\right),
\end{equation}
where $b_{11}\neq 0$.

The following theorem is a main result of the present section.

\begin{theorem} \label{pi_3 Lie_group}
The following relation takes place
\[
LocAut_+(\pi_3)=\{\exp(\nabla): \nabla\in LocDer(\pi_3)\},
\]
where $LocDer(\pi_3)$ is the Lie algebra of all local derivations of the algebra $\pi_3$ and
$\exp(\nabla)=I+\sum_{n=1}^\infty \frac{\nabla^n}{n!}$, $I$ is the identity mapping of $\pi_3$.
\end{theorem}

\begin{proof}
Let $\nabla$ be a local derivation of algebra $\pi_3$, and, let $A$ be the matrices of $\nabla$ and
\[
A=
\left(
  \begin{array}{ccccc}
    x_{1} & 0 & 0 & 0 & 0 \\
    x_{4} & 2x_{1} & 0 & 0 & 0 \\
    x_{6} & x_{5} & 3x_{1} & x_{2} & 0 \\
    0& 0 & 0 &   x_{1} & 0 \\
    x_{7} & 0 & 0 & x_{3} & 2x_{1} \\
  \end{array}
\right).
\]
Then
\[
A^2=A*A=
\left(
  \begin{array}{ccccc}
    x_{1}^2 & 0 & 0 & 0 & 0 \\
    3x_{1}x_{4} & 4x_{1}^2 & 0 & 0 & 0 \\
    4x_{1}x_{6}+x_{4}x_{5} & 5x_{1}x_{5} & 9x_{1}^2 & 4x_{1}x_{2} & 0 \\
    0& 0 & 0 &   x_{1}^2 & 0 \\
    3x_{1}x_{7} & 0 & 0 & 3x_{1}x_{3} & 4x_{1}^2 \\
  \end{array}
\right),
\]
\[
A^3=A*A^2=
\left(
  \begin{array}{ccccc}
    x_{1}^3 & 0 & 0 & 0 & 0 \\
    7x_{1}^2x_{4} & 8x_{1}^3 & 0 & 0 & 0 \\
    13x_{1}^2x_{6}+6x_{1}x_{4}x_{5} & 19x_{1}^2x_{5} & 27x_{1}^3 & 13x_{1}^2x_{2} & 0 \\
    0& 0 & 0 &   x_{1}^3 & 0 \\
    7x_{1}^2x_{7} & 0 & 0 & 7x_{1}^2x_{3} & 8x_{1}^3 \\
  \end{array}
\right),
\]
\[
A^4=A*A^3=
\left(
  \begin{array}{ccccc}
    x_{1}^4 & 0 & 0 & 0 & 0 \\
    15x_{1}^3x_{4} & 16x_{1}^4 & 0 & 0 & 0 \\
    40x_{1}^3x_{6}+25x_{1}^2x_{4}x_{5} & 65x_{1}^3x_{5} & 81x_{1}^4 & 40x_{1}^3x_{2} & 0 \\
    0& 0 & 0 &   x_{1}^4 & 0 \\
    15x_{1}^3x_{7} & 0 & 0 & 15x_{1}^3x_{3} & 16x_{1}^4 \\
  \end{array}
\right)
\]
and
\[
A^5=A*A^4=
\left(
  \begin{array}{ccccc}
    x_{1}^5 & 0 & 0 & 0 & 0 \\
    31x_{1}^4x_{4} & 32x_{1}^5 & 0 & 0 & 0 \\
    121x_{1}^4x_{6}+90x_{1}^3x_{4}x_{5} & 211x_{1}^4x_{5} & 243x_{1}^5 & 121x_{1}^4x_{2} & 0 \\
    0& 0 & 0 &   x_{1}^5 & 0 \\
    31x_{1}^4x_{7} & 0 & 0 & 31x_{1}^4x_{3} & 32x_{1}^5 \\
  \end{array}
\right).
\]
Now we compute the components of the matrix
$B=\exp(A)=E+\sum_{n=1}^\infty \frac{A^n}{n!}$, where $E$ is the unit matrix. Let $B=(b_{i,j})_{i,j=1}^5$.
Since, for $x\in {\Bbb F}$,
\[
e^x=1+x+\frac{x^2}{2!}+\frac{x^3}{3!}+\frac{x^4}{4!}+\dots+\frac{x^5}{n!}+\dots,
\]
 we have
\[
b_{1,1}=e^{x_{1}},
b_{2,2}=e^{2x_{1}},
b_{3,3}=e^{3x_{1}},
b_{4,4}=e^{x_{1}},
b_{5,5}=e^{2x_{1}}.
\]
Now we have $b_{1,2}=0$, $b_{1,3}=0$, $b_{1,4}=0$, $b_{1,5}=0$,
$b_{2,1}=\lambda_{2,1}(x_1)x_{4}$, where
\[
\lambda_{2,1}(x_1)=1+\frac{(2^0+2^1)x_{1}}{2!}+\frac{(2^0+2^1+2^2)x_{1}^2}{3!}+\frac{(2^0+2^1+2^2+2^3)x_{1}^3}{4!}
\]
\[
+\frac{(2^0+2^1+2^2+2^3+2^4)x_{1}^4}{5!}+\dots+\frac{(\sum_{k=0}^n2^k)x_1^n}{(n+1)!}+\dots,
\]
We have
\[
\lim_{n\to\infty}\frac{(\sum_{k=0}^n2^k)x_1^n}{(n+1)!}=0
\]

\[
b_{2,3}=0, b_{2,4}=0, b_{2,5}=0, b_{2,4}=0, b_{2,5}=0, b_{3,1}=\lambda_{3,1}(x_1)x_{6}+\mu_{3,1}(x_1)x_{4}x_{5},
\]
where
\[
\lambda_{3,1}(x_1)=1+\frac{(1+3)x_{1}}{2!}+\frac{(1+3(1+3))x_{1}^2}{3!}+\frac{(1+3(1+3(1+3)))x_{1}^3}{4!}+\dots
\]
\[
+\frac{(1+3_1(1+3_2(1+3_3(1+3_4(...(1+3_{n-1})...)))))x_{1}^{n-1}}{n!}+\dots,
\]
\[
\mu_{3,1}(x_1)=\frac{1}{2!}+\frac{((1+2)+3)x_{1}}{3!}+\frac{((1+2+2^2)+3((1+2)+3))x_{1}^2}{4!}
\]
\[
+\frac{((1+2+2^2+2^3)+3((1+2+2^2)+3((1+2)+3)))x_{1}^3}{5!}+\dots
\]
\[
+\frac{((1+2+2^2+...+2^{n-2})+3((1+2+2^2+...+2^{n-3})+\dots+3((1+2)+3)))x_{1}^{n-2}}{n!}+\dots,
\]
$b_{3,2}=\lambda_{3,2}(x_1)x_{5}$, where
\[
\lambda_{3,2}(x_1)=1+\frac{(2+3)x_{1}}{2!}+\frac{(2^2+3(2+3))x_{1}^2}{3!}+\frac{(2^3+3(2^2+3(2+3)))x_{1}^3}{4!}
\]
\[
+\frac{(2^4+3(2^3+3(2^2+3(2+3))))x_{1}^4}{5!}+\dots
\]
\[
+\frac{(2^{n-1}+3_1(2^{n-2}+3_2(2^{n-3}+3_3(+\dots+3_{n-1}))))x_{1}^{n-1}}{n!}+\dots,
\]
$b_{3,4}=\lambda_{3,4}(x_1)x_{2}$, where
\[
\lambda_{3,4}(x_1)=1+\frac{(1+3)x_{1}}{2!}+\frac{(1+3(1+3))x_{1}^2}{3!}+\frac{(1+3(1+3(1+3)))x_{1}^3}{4!}+\dots
\]
\[
+\frac{(1+3_1(1+3_2(1+3_3(1+3_4(...(1+3_{n-1})...)))))x_{1}^{n-1}}{n!}+\dots,
\]
\[
b_{3,5}=0, b_{4,1}=0, b_{4,2}=0, b_{4,3}=0, b_{4,5}=0,
\]
\[
b_{5,1}=\lambda_{2,1}(x_1)x_{7}, b_{5,2}=0, b_{5,3}=0, b_{5,4}=\lambda_{2,1}(x_1)x_{3}.
\]

From these it follows that the matrix
\begin{equation}
B=
\left(
  \begin{array}{ccccc}
    b_{11} & 0 & 0 & 0 & 0 \\
    b_{21} & b_{11}^2 & 0 & 0 & 0 \\
    b_{31} & b_{32} & b_{11}^3 & b_{34} & 0 \\
    0 & 0 & 0 &   b_{11} & 0 \\
    b_{51} & 0 & 0 & b_{54} & b_{11}^2 \\
  \end{array}
\right)
\end{equation}
is of the form \ref{(137)}. By Theorem \ref{3.3}, the linear operator $\Phi$ generated by the matrix $B$ is
a local automorphism. Therefore
\[
\{\exp(\nabla): \nabla\in LocDer(\pi_3)\}\subseteq LocAut(\pi_3)
\]

Conversely, let $\Phi$ be a local automorphism from $LocAut_+(\pi_3)$, and, let $B$ be the matrix of $\Phi$ and
\begin{equation}
B=
\left(
  \begin{array}{ccccc}
    b_{11} & 0 & 0 & 0 & 0 \\
    b_{21} & b_{11}^2 & 0 & 0 & 0 \\
    b_{31} & b_{32} & b_{11}^3 & b_{34} & 0 \\
    0 & 0 & 0 &   b_{11} & 0 \\
    b_{51} & 0 & 0 & b_{54} & b_{11}^2 \\
  \end{array}
\right).
\end{equation}
We should find a matrix
\[
A=
\left(
  \begin{array}{ccccc}
    x_{1} & 0 & 0 & 0 & 0 \\
    x_{4} & 2x_{1} & 0 & 0 & 0 \\
    x_{6} & x_{5} & 3x_{1} & x_{2} & 0 \\
    0& 0 & 0 &   x_{1} & 0 \\
    x_{7} & 0 & 0 & x_{3} & 2x_{1} \\
  \end{array}
\right)
\]
such that
\[
b_{1,1}=e^{x_{1}}, b_{2,1}=\lambda_{2,1}(x_1)x_{4},
\]
\[
b_{3,1}=\lambda_{3,1}(x_1)x_{6}+\mu_{3,1}(x_1)x_{4}x_{5},
\]
\[
b_{3,2}=\lambda_{3,2}(x_1)x_{5}, b_{3,4}=\lambda_{3,4}(x_1)x_{2},
\]
\[
b_{5,1}=\lambda_{2,1}(x_1)x_{7}, b_{5,4}=\lambda_{2,1}(x_1)x_{3}.
\]
We consider the following system of equations
\begin{equation}\label{(7)}
\left\{%
\begin{array}{lllllll}
    {b_{1,1}=e^{x_{1}}} \\
    {b_{2,1}=\lambda_{2,1}(x_1)x_{4}} \\
    {b_{3,1}=\lambda_{3,1}(x_1)x_{6}+\mu_{3,1}(x_1)x_{4}x_{5}} \\
    {b_{3,2}=\lambda_{3,2}(x_1)x_{5}} \\
    {b_{3,4}=\lambda_{3,4}(x_1)x_{2}} \\
    {b_{5,1}=\lambda_{2,1}(x_1)x_{7}} \\
    { b_{5,4}=\lambda_{2,1}(x_1)x_{3}}
\end{array}%
\right.
\end{equation}
with respect to the variables $x_{1}$, $x_{2}$, $\dots$, $x_{7}$.
Clearly, $x_{1}=\ln b_{1,1}$.

Now, since the infinite sum
\[
1+x_{1}+\frac{x_{1}^2}{2!}+\frac{x_{1}^3}{3!}+\frac{x_{1}^4}{4!}+\frac{x_{1}^5}{5!}+\dots=e^{x_{1}}
\]
is bounded, we have the infinite sums $\lambda_{2,1}(x_1)$, $\lambda_{3,1}(x_1)$, $\mu_{3,1}(x_1)$,
$\lambda_{3,2}(x_1)$, $\lambda_{3,4}(x_1)$, $\lambda_{5,1}(x_1)$ and $\lambda_{5,4}(x_1)$ are also
bounded.

Now we select values for the remaining variables in the following order: $x_{4}$, $x_{5}$, $x_{6}$, $x_{2}$, $x_{7}$, $x_{3}$, i.e.,
\begin{equation}
\left\{%
\begin{array}{lllllll}
    {b_{1,1}=e^{x_{1}}} \\
    {b_{2,1}=\lambda_{2,1}(x_1)x_{4}} \\
    {b_{3,2}=\lambda_{3,2}(x_1)x_{5}} \\
    {b_{3,1}=\lambda_{3,1}(x_1)x_{6}+\mu_{3,1}(x_1)x_{4}x_{5}} \\
    {b_{3,4}=\lambda_{3,4}(x_1)x_{2}} \\
    {b_{5,1}=\lambda_{2,1}(x_1)x_{7}} \\
    {b_{5,4}=\lambda_{2,1}(x_1)x_{3}}
\end{array}%
\right.
\end{equation}

Thus we find a matrix $A$ such that $B=\exp(A)$. This denotes $\Phi=\exp(\nabla)=I+\sum_{n=1}^\infty \frac{\nabla^n}{n!}$,
where $\nabla$ is a local derivation generated by the matrix $A$. The proof is complete.
\end{proof}

Note that the group of all matrices of local automorphisms corresponding to the group $LocAut(\pi_3)$
is not topologically closed. Therefore we can not apply Cartan's theorem (Closed subgroup theorem) to
the group $LocAut(\pi_3)$.

The following theorem is proved similar to the proof of Theorem \ref{pi_3 Lie_group}.

\begin{theorem}
The group $LocAut(\pi_2)$ of all local automorphisms of algebra $\pi_2$ satisfies the following condition
\[
LocAut(\pi_2)=\{\exp(\nabla): \nabla\in LocDer(\pi_2)\},
\]
where $LocDer(\pi_2)$ is the Lie algebra of all local derivations of algebra $\pi_2$ and
$\exp(\nabla)=I+\sum_{n=1}^\infty \frac{\nabla^n}{n!}$, $I$ is the identity mapping of $\pi_2$.
\end{theorem}

Similar to the case of algebra $\pi_3$, since the group of all matrices of local automorphisms
corresponding to the group $LocAut(\pi_2)$ is not topologically closed, we can not apply Cartan's theorem
to group $LocAut(\pi_2)$.

Let us denote the set of all matrices of the form (\ref{pi_22}) by $\mathcal{M}(LocAut(\pi_2))$.
$\mathcal{M}(LocAut(\pi_2))$ is a set of matrices in which some elements are expressed
as linear functions of others, and the rest are fixed or zero.

Let us define a mapping:
\[
\phi: \mathbb{R}^{11} \to \mathbb{R}^{25},\quad \text{ mapping } \vec{b} \mapsto B,
\]
where $\vec{b}=(b_{11}, b_{21}, b_{22}, b_{31}, b_{32}, b_{33}, b_{34}, b_{41}, b_{51}, b_{52}, b_{54})$,
$B$ is a matrix whose elements are determined by the formula given in (\ref{pi_22}).
Then $\mathcal{M}(LocAut(\pi_2))=\phi(\mathbb{R}^{11})$, that is, the image of a smooth mapping $\phi$, where:
\begin{itemize}
\item $\phi$ is a smooth (even linear or affine) mapping.
\item Space $\mathbb{R}^{11}$ is an Euclidean space.
\end{itemize}

Let's verify that $\phi$ is an embedding:
\begin{itemize}
\item The matrix $B$ is uniquely determined by the parameters $b_{ij}$, that is, the mapping $\phi$ is injective.
\item The dependencies between the coordinates are linear: $b_{41}+b_{11}$, $b_{22}+b_{52}$ are smooth functions.
\item The differential $D\phi$ has maximum rank $11$, since the mapping is linearly independent with respect to all $11$ parameters.
\end{itemize}

Since a smooth mapping $\phi: \mathbb{R}^{11} \to \mathbb{R}^{25}$ is an embedding (an injective mapping with regular differential),
we have the image $\mathcal{M}(LocAut(\pi_2))$ is a smooth submanifold of dimension $11$.
Thus, the set $\mathcal{M}(LocAut(\pi_2))$ is a smooth embedded submanifold of $\mathbb{R}^{25}$ (the space of all $5\times 5$ matrices), of dimension $11$.
As the result we have the following theorem.

\begin{theorem}
The group $LocAut(\pi_2)$ of all local automorphisms of algebra $\pi_2$ is a Lie group.
\end{theorem}

Let us denote the set of all matrices of the form (\ref{(136)})
by $\mathcal{M}(LocAut(\pi_3))$. The question arises: is this set a smooth manifold in $\mathbb{R}^{25}$?

Note that:
\begin{itemize}
  \item The matrix depends on the variables $b_{11}$, $b_{21}$, $b_{31}$, $b_{32}$, $b_{34}$, $b_{51}$, $b_{54}$,
that is, $7$ free parameters.
  \item The remaining elements are expressed through $b_{11}$ or are equal to $0$.
  \item The sign in $\pm b_{11}^3$ gives two possible options.
\end{itemize}

Thus, the set of all such matrices is the union of two sets corresponding to the plus and minus signs in $\pm b_{11}^3$
\[
\mathcal{M}(LocAut(\pi_3))=\mathcal{M}(LocAut_{+}(\pi_3)) \cup \mathcal{M}(LocAut_{-}(\pi_3)),
\]
where
$\mathcal{M}(LocAut_{+}(\pi_3))$ is a subset of matrices with $+b_{11}^3$,
$\mathcal{M}(LocAut_{-}(\pi_3))$ is a subset of matrices with $-b_{11}^3$.

Consider the mapping
\[
\Phi_\pm: \mathbb{R}^7 \supset U \to \mathbb{R}^{5 \times 5},
\]
which, for each set of parameters
\[
(b_{11}, b_{21}, b_{31}, b_{32}, b_{34}, b_{51}, b_{54}) \mapsto B,
\]
forms a matrix of the form (\ref{(136)}) with a fixed sign in $\pm b_{11}^3$.
This mapping is smooth (all components are polynomials in the parameters), and its image
is a smooth submanifold of dimension $7$ in $\mathbb{R}^{25}$, since the mapping is an embedding.
Thus, each of the sets $\mathcal{M}(LocAut_{+}(\pi_3))$ and $\mathcal{M}(LocAut_{-}(\pi_3))$
is a smooth $7$-dimensional manifold in $\mathbb{R}^{25}$.

The set $\mathcal{M}(LocAut(\pi_3))=\mathcal{M}(LocAut_{+}(\pi_3)) \cup \mathcal{M}(LocAut_{-}(\pi_3))$
is the union of two disjoint smooth manifolds.
Note that they are disjoint because if $b_{11}\ne 0$, then $b_{11}^3 \ne -b_{11}^3$, which means $\pm b_{11}^3$
are distinct. Hence, matrices cannot simultaneously lie in both sets.

The union of two disjoint smooth submanifolds is a smooth manifold,
Therefore, the set of such matrices is a smooth manifold.
Thus, the following theorem holds true.

\begin{theorem}
The group $LocAut(\pi_3)$ of all local automorphisms of algebra $\pi_3$ is a Lie group.
\end{theorem}

\section{Conclusion}

Note that the general form of the matrix of a local automorphism (local derivation) on an algebra includes the general form of the matrix of an automorphism (resp. derivation) on this algebra. The coincidence of these general forms denotes that every local automorphism (local derivation) of the considering algebra is an automorphism (resp. derivation). But the general form of the matrix of an automorphism (derivation) on nilpotent associative algebra $\pi_2$ and $\pi_3$ does not coincide with the general form of the matrix of a local automorphism (resp. local derivation) on these algebras by Theorems \ref{2.4}, \ref{2.5}, \ref{3.2}, \ref{3.3} (resp. by Theorems \ref{Der_pi_2}, \ref{Der_pi_3}, \ref{LocDer_pi_2}, \ref{LocDer_pi_3}). Therefore, nilpotent associative algebras $\pi_2$ and $\pi_3$ have local automorphisms (local derivations) that are not automorphisms (resp derivations).
Therefore, the sets of local automorphisms $LocAut(\pi_2)$ and $LocAut(\pi_3)$ form groups distinct from the groups $Aut(\pi_2)$ and $Aut(\pi_3)$
respectively.

The vector space $LocDer(\pi_2)$ ($LocDer(\pi_3)$) of all local derivations of algebra $\pi_2$ (resp. $\pi_3$) forms a Lie algebra with respect
to the Lie multiplication $[\nabla,\Delta]=\nabla\Delta-\Delta\nabla$.

For every local derivation $\nabla\in LocDer(\pi_2)$ the mapping $\exp(\nabla)$ is a local automorphism of $\pi_2$.
For every local derivation $\nabla\in LocDer(\pi_3)$ the mapping $\exp(\nabla)$ is a local automorphism from $LocAut_+(\pi_3)$ and,
conversely, for every local automorphism $\Phi\in LocAut_+(\pi_3)$ there exists a local derivation $\nabla\in LocDer(\pi_3)$
such that $\Phi=\exp(\nabla)$. The groups $LocAut(\pi_2)$ and $LocAut(\pi_3)$ of local automorphisms of algebras $\pi_2$ and $\pi_3$
respectively are Lie groups.

\section{Statements and Declarations}

Ethical Approval Not Applicable.

No funding was received to assist with the preparation of this manuscript.

The authors have no competing interests to declare that are relevant to the content of this article.


\begin{thebibliography}{99}

\bibitem{1}
Arzikulov F.N., Karimjanov I.A., Umrzaqov S.M. (2022) Local and 2-local automorphisms of some solvable Leibniz algebras. {\it Journal of Geometry and Physics} 178: 104573 DOI: 10.1016/j.geomphys.2022.104573

\bibitem{1.1}
Arzikulov F.N., Qurbonov A. (2025) On automorphisms of finite-dimensional naturally graded associative algebras. {\it Scientific Bulletin of Namangan State University} No 7: 38--42.

\bibitem{AEK}
Ayupov Sh.A., Elduque A., Kudaybergenov K.K. (2023) Local derivations and automorphisms of Cayley algebras. {\it Journal of Pure and Applied Algebra} 227(5): 107277. DOI: 10.1016/j.jpaa.2022.107277

\bibitem{2}
Ayupov Sh.A., Kudaybergenov K.K. (2017) Local Automorphisms on Finite-Dimensional Lie and Leibniz Algebras. Algebra, Complex Analysis and Pluripotential Theory, USUZCAMP. {\it Springer Proceedings in Mathematics and Statistics} 264: 31--44. DOI: 10.1007/978-3-030-01144-4$_-$3

\bibitem{3}
Ayupov Sh.A., Kudaybergenov K.K., Omirov B.A. (2020) Local and 2-local derivations and automorphisms on simple Leibniz algebras. {\it Bulletin of the Malaysian Mathematical Sciences Society} 43: 2199--2234. DOI: 10.1007/s40840-019-00799-5

\bibitem{5}
Becker T., Escobar J., Salas C., Turdibaev R. (2019) On Local Automorphisms of $\mathfrak{sl}_2$. {\it Uzbek Mathematical Journal} No 2: 27--34. DOI: 10.29229/uzmj.2019-2-3

\bibitem{6}
Costantini M. (2019) Local automorphisms of finite dimensional simple Lie algebras. {\it Linear Algebra Appl.} 562: 123--134. DOI: 10.1016/j.laa.2018.10.009

\bibitem{EZLS}
Elisova A.P., Zotov I.N., Levchuk V.M., Suleimanova G.S. (2011) Local automorphisms and local derivations of nilpotent matrix algebras. {\it Bulletin of Irkutsk State University. Series Mathematics} 4(1): 9--19. https://mathnet.ru

\bibitem{7}
Johnson B. (2001) Local derivations on C$^*$-algebras are derivations. {\it Trans. Amer. Math. Soc.} 353: 313--325. DOI: 10.2307/221975

\bibitem{8}
Kadison R. (1990) Local derivations. {\it J. Algebra} 130: 494--409. DOI: 10.1016/0021-8693(90)90095-6

\bibitem{KL}
Karimjanov I.A., Ladra M. (2020) Some classes of nilpotent associative algebras. {\it Mediterr. J. Math.} 17(2): 1--21. https://doi.org

\bibitem{J-HL-NCW}
Liu J.-H., Wong N.-Ch. (2007) Local automorphisms of operator algebras. {\it Taiwanese Journal of Mathematics} 11(3): 611--619. DOI: 10.11650/twjm/1500404747

\bibitem{9}
Wang Y., Chen H., Nan J. (2019) 2-Local automorphisms on basic classical Lie superalgebras. {\it Acta Mathematica Sinica, English Series} 35(3): 427--437. DOI: 10.1007/s10114-018-7519-6

\end{thebibliography}
\end{document}